\documentclass[a4paper,12pt]{article}

\usepackage{amsmath}
\usepackage{amssymb}
\usepackage{amsthm}
\usepackage{bbm}

\newtheorem{theorem}{Theorem}[section]
\newtheorem{lemma}[theorem]{Lemma}
\newtheorem{proposition}[theorem]{Proposition}

\theoremstyle{definition}
\newtheorem{example}[theorem]{Example}
\newtheorem{definition}[theorem]{Definition}

\newcommand{\R}{{\mathbb{R}}}

\newcommand{\Q}{{\mathbb{Q}}}

\newcommand{\Z}{{\mathbb{Z}}}
\newcommand{\N}{{\mathbb{N}}}

\newcommand{\Id}{{\mathbbm{1}}}
\newcommand{\CPn}{{\mathbb{C} \mathbb{P}^n}}

\newcommand{\bg}{{\bar{g}}}
\newcommand{\Cal}{{\hbox{\it Cal\,}}}

\newcommand{\Ham}{{\hbox{\it Ham\,}}}

\newcommand{\tHam}{\widetilde{\hbox{\it Ham}\, }}

\newcommand{\Poincare}{Poincar\'e}

\begin{document}

\title{Hofer Growth of $C^1$-generic Hamiltonian flows }

\author{Asaf Kislev\thanks{Partially supported by European Research Council advanced grant 338809}}

\maketitle
\begin{abstract}
We prove that on certain closed symplectic manifolds a $C^1$-generic cyclic subgroup of the universal cover of the group of Hamiltonian diffeomorphisms is undistorted with respect to the Hofer metric.
\end{abstract}

\section{Introduction}

\subsection{Hofer growth of cyclic subgroups}

Let $(M^{2n},\omega)$ be a closed symplectic manifold. We denote by $\Ham(M)$ the group of Hamiltonian diffeomorphisms and by $\tHam(M)$ its universal cover.
For elements in $\tHam(M)$ we use Greek letters and for elements in $\Ham(M)$ Roman letters. For instance, we will write $\phi \in \tHam(M)$ or $[ \{f_t\}_{t\in [0,1]}] \in \tHam(M) $, where $\{f_t\} \subset \Ham(M)$ is a smooth path of Hamiltonian diffeomorphisms with $f_0 = \Id$, and $[ \{f_t\}_{t\in [0,1]}]$ stands for the homotopy class with fixed end points.
When we write $f$ with no subscript we are referring to the time-1-map $f=f_1$.

The Hofer metric on $\Ham(M)$ is defined by
\[ d(g,f) = \inf(\int_0^1 \max|H_t| dt),\]
where the infimum is taken over all Hamiltonian functions $H$ which generate $f g^{-1}$ as their time-1-map.

We denote the lift of the Hofer metric to $\tHam(M)$ also by $d$, i.e.
\[d(\psi,\phi) = \inf(\int_0^1 \max|H_t| dt),\]
where the infimum is taken over all Hamiltonian functions $H$ whose Hamiltonian flow is in class $\phi \psi^{-1}$.

Let $\{f^n\}_{n \in \Z} \subset \Ham(M)$ be a cyclic subgroup. We say that $\{f^n\}_{n \in \Z}$ is \begin{it}undistorted\end{it} if
\[\lim_{n \to \infty} \frac{d(\Id, f^n)}{n}  > 0.\]
Note that the limit always exists because $d(\Id, f^n)$ is a subadditive sequence.
Similarly for a cyclic subgroup $\{\phi^n \}_{n \in \Z} \subset \tHam(M)$ , we say that $\{\phi^n \}_{n \in \Z}$ is \begin{it}undistorted\end{it} if
\[\lim_{n \to \infty} \frac{d(\Id, \phi^n)}{n}  > 0.\]

The distortion of subgroups of $\tHam(M)$ has been studied on various occasions in connection to Hamiltonian dynamics and ergodic theory, see e.g. \cite[chapters 8,11]{2}.

In the autonomous case, it has been proved that there exists a $C^0$-open and $C^\infty$-dense subset $\mathcal{A}$ of the set of autonomous normalized Hamiltonian functions such that for every $F \in \mathcal{A}$, the cyclic subgroup generated by the Hamiltonian flow of $F$ is undistorted (see \cite[chapter 6]{3}).

In this article we give a similar statement for a $C^1$-generic time dependent element in $\tHam(M)$.
When we say \begin{it}$C^1$-generic\end{it} we mean that the set of elements in $\tHam(M)$ that generate undistorted cyclic subgroups has a $C^1$-open and dense subset.

Let us recall the definition of the $C^1$ topology on $\tHam(M)$.
It is known (see \cite{1}) that $\Ham(M)$ is locally simply connected.
Fix a basis $\{U_i\}$ of simply connected $C^1$-neighborhoods of $\Id$ in $\Ham(M)$.
Let $\tilde{U}_i$ be the lift of $U_i$ to $\tHam(M)$ that contains $\Id \in \tHam(M)$.
By definition, the sets $\{\phi \tilde{U}_i\}$ form a basis of $C^1$-neighborhoods of $\phi \in \tHam(M)$.

\textbf{Digression on Chern classes:}
$H^i(M;\Q)$ contains a lattice $H^i(M;\Z) / \text{torsion}$ whose elements are called \begin{it}integral classes\end{it}.
The Chern classes with rational coefficients are by definition integral.
In what follows we abbreviate $H^i(M) := H^i(M;\Q)$.

Before we state the main results of this paper, let us give the following definitions.

\begin{definition}

Let
\[ \alpha \in H^*(M) := \oplus_i H^i(M) .\]
When we write $\deg(\alpha)$ we mean the maximal $k$ such that the projection of $\alpha$ to $H^{2k}(M)$ is non-zero.

\end{definition}
\begin{definition}
Let $M^{2n}$ be a closed symplectic manifold. Let $c(M) \in H^*(M)$ be the full Chern class of $TM$.
We say that there exists an \begin{it}even factorization\end{it} of $c(M)$ if we can write
\[ c(M) = \alpha \beta,\]
where
\[ \deg(\alpha) + \deg(\beta) \leq n,\]
\[ 0 < \deg(\alpha) < n,\]
the classes $\alpha$ and $\beta$ are integral classes,
and $\alpha$ has only terms of even degree.
In this case we say that $c(M) = \alpha \beta$ is an even factorization.
When we say that $\alpha$ has only terms of even degree we mean that the projection of $\alpha$ to $H^{2k}(M)$ is zero for odd $k$.
\end{definition}

Our motivation for the above deﬁnition is that symplectic manifolds with an even factorization cannot admit partially hyperbolic Hamiltonian diffeomorphisms (see Section 3). This has been announced by Bennequin and is detailed in \cite[pp. 137--138]{6}.

We are now ready to state the main theorem.

\begin{theorem}
\label{t5}
Let $M$ be a closed symplectic manifold with $H^1(M) = 0$. If the top Chern class $c_n(M) \neq 0$ and there is no even factorization of the full Chern class,
then the set of elements in $\tHam(M)$ that generate undistorted cyclic subgroups is $C^1$-generic.
\end{theorem}

The following theorem is a corollary (for its proofs see Example \ref{e2}).
\begin{theorem}
\label{t10}
Let $M^4$ be a four dimensional closed symplectic manifold with $H^1(M) = 0$, and $c_2(M) \neq 0$.
Then the set of elements in $\tHam(M)$ that generate undistorted cyclic subgroups is $C^1$-generic.

\end{theorem}

In particular, for $M = \mathbb{C}\mathbb{P}^2$, $C^1$-generic elements generate undistorted cyclic subgroups. In fact  for $\CPn$ we can upgrade the theorem and formulate it with respect to cyclic subgroups of $\Ham(M)$.

\begin{theorem}
\label{t7}
The set of Hamiltonian diffeomorphisms in $\Ham(\CPn)$ that generate undistorted cyclic subgroups of $\Ham(\CPn)$ has a $C^1$-open and dense subset.
\end{theorem}

\subsection{Idea of the proof}
Consider the set
\[ \chi = \{ \phi \in \tHam(M) : \sigma(\phi) \neq 0 \},\]
where $\sigma$ is an asymptotic spectral invariant (see Section 1.4 for the definition).
We will show that $\chi$ is a subset of the set of elements in $\tHam(M)$ that generate undistorted cyclic subgroups (see Proposition \ref{p1}), and that in certain manifolds $\chi$ is a $C^1$-open and dense set in $\tHam(M)$.

In Section 2 we prove that $\chi$ is $C^1$-open.

In our proof that $\chi$ is $C^1$-dense we restrict to the case where the set of elements in $\tHam(M)$ that have an elliptic periodic point is $C^1$-dense.
In Section 3 we give a method to check that this is the case by examining the full Chern class of $TM$. We show that if there is no even factorization then the set of elements in $\tHam(M)$ that have an elliptic periodic point is $C^1$-dense. First we follow Bennequin (see \cite{6}) and show that if there is no even factorization then there are no partially hyperbolic symplectomorphisms. Next we use a result by Saghin and Xia (\cite{4}) which states that a $C^1$-generic symplectomorphism which is not partially hyperbolic has an elliptic periodic point. (The definitions of elliptic periodic points and partially hyperbolic symplectomorphisms are given in the next subsection.)

In Section 4 we deal with the last step of the proof, which is to show that for every $\phi \in \tHam(M)$ with an elliptic periodic point, we can $C^1$-perturb $\phi$ to $\tilde \phi \in \widetilde Ham (M)$ with $\tilde \phi \in \chi$. For the main part of the construction of the perturbation we follow Bonnati, Crovisier, Vago and Wilkinson (\cite{5}). This shows that if the set of elements in $\tHam(M)$ that have an elliptic periodic point is $C^1$-dense, then $\chi$ is $C^1$-dense in $\tHam(M)$.

In Section 5 we discuss whether our results can be applied to $\Ham(M)$, that is whether a $C^1$-generic element in $\Ham(M)$ generates an undistorted cyclic subgroup with respect to the Hofer metric.

In Section 6 we give examples of manifolds that satisfy the requirements on the full Chern class. For these manifolds a $C^1$-generic element in $\tHam(M)$ generates an undistorted cyclic subgroup.

\subsection{Partially hyperbolic maps and elliptic periodic points}
\label{s1}

Let $M^{2n}$ be a closed symplectic manifold such that $H^1(M) = 0$, and let $f \in Ham(M)$.
In this section and also throughout the paper we assume that an auxiliary Riemannian metric has been chosen.

A point $p \in M$ is called an \begin{it}elliptic $l$-periodic point\end{it} if $f^l(p) = p$ and all of the eigenvalues of $d_p (f^l)$ are simple, not real and of norm 1.

A continuous splitting of the tangent bundle $TM = A \oplus B$ is called \begin{it}invariant\end{it} if it is invariant under $df$.
For an invariant splitting we say that $A$ \begin{it}dominates\end{it} $B$ if there exists $m > 0$ such that for every $x \in M$ and any two unit vectors $u \in A_x, v \in B_x$, the following inequality holds
\[ \| d_x f^m (u) \| \geq 2 \| d_x f^m (v) \|. \]

A diffeomorphism $f$ is called \begin{it}partially hyperbolic\end{it} if the following conditions hold:
\begin{enumerate}
\item There is an invariant splitting $TM = E^u \oplus E^c \oplus E^s$ with at least two of them non-trivial.
\item $E^u$ is uniformly expanding, i.e. there exist $\alpha > 1$ and $a > 0$ such that $ \| d f^k (v) \| \geq a \alpha^k \|v\| $ for all $v \in E^u, k \in \N$.
\item  $E^s$ is uniformly contracting, i.e. there exist $\beta > 1$ and $b > 0$ such that $ \| d f^{-k} (u) \| \geq b \beta^{k} \|u\| $ for all $u \in E^s, k \in \N$.
\item $E^u$ dominates $E^c$, and $E^c$ dominates $E^s$.
\end{enumerate}

If $f$ is a partially hyperbolic symplectomorphism, then one can choose the splitting so that $rank(E^u) = rank(E^s)$ (see \cite[Lemma 8]{4}).

Let $Symp_1(M)$ be the set of all $C^1$ symplectomorphisms, and
\[Symp(M) \subset Symp_1(M)\]
 be the set of all $C^\infty$ symplectomorphisms.
A result by Saghin and Xia states that there exists a $C^1$-open and dense $\mathcal{U} \subset Symp_1(M)$ such that every $f \in \mathcal{U}$ is either partially hyperbolic or has an elliptic periodic point, see \cite{4}.
See also \cite{13} for a similar and independent result by Horita and Tahzibi.

Note that since $\mathcal{U}$ is open, $\mathcal{U} \cap Symp(M)$ is $C^1$-dense in $Symp(M)$. We get that in a $C^1$-dense subset, a $C^\infty$ symplectomorphism which is not partially hyperbolic has an elliptic periodic point. In the case where $H^1(M) = 0$, the result is also true for a $C^1$-dense subset of the set of Hamiltonian diffeomorphisms. This is a simple consequence of the fact that the group of symplectomorphisms is locally path connected and the subgroup of Hamiltonian diffeomorphisms is exactly the connected component of the identity.
Since $\pi : \tHam(M) \to \Ham(M)$ is open and continuous, we get the following
\begin{theorem}
\label{t8}
The set
\[\{ [\{f_t\}] \in \tHam(M) : \begin{matrix} f_1 \text{ is partially hyperbollic or } \\ f_1 \text{ has an elliptic periodic point} \end{matrix}\}\]
is $C^1$-dense in $\tHam(M)$.

\end{theorem}

\subsection{Asymptotic spectral invariants}
\begin{definition}
Let $(U^{2n},\omega)$ be an open symplectic manifold. Let $\phi \in \tHam(U)$ be an element such that there is a representative generated by a compactly supported Hamiltonian function $\{F_t\}_{t \in [0,1]}$. We define the \begin{it} Calabi invariant\end{it} of $\phi$ to be
\[ \Cal (\phi) = \int_0^1 \int_U F_t \omega^n \; dt. \]

\end{definition}

It is known that the Calabi invariant is well defined and it defines a homomorphism from $\tHam(U)$ to $\R$ (see \cite{1}).

\begin{definition}
 A function $c:\tHam(M) \to \R$ is called a \begin{it}subadditive spectral invariant\end{it} if
  \begin{enumerate}
  \item (conjugation invariance) $\forall \phi, \psi \in \tHam(M), c(\phi \psi \phi^{-1}) = c(\psi)$.
  \item (subadditivity) $c(\phi \psi) \leq c(\phi) + c(\psi)$
  \item (stability) $\int_0^1 \min(F_t - G_t) dt \leq c(\phi) - c(\psi) \leq \int_0^1 \max(F_t - G_t) dt ,$
  where $\phi$ and $\psi$ have representatives that are generated by compactly supported Hamiltonian functions $F$ and $G$ respectively.
  \item (spectrality) $c(\phi) \in spec(\phi)$ for all non-degenerate $\phi \in \tHam(M)$.

  \end{enumerate}
\end{definition}

Recall that an element $[\{f_t\}] \in \tHam(M)$ is called non-degenerate if the graph of $f_1$ in $M \times M$ is transverse to the diagonal. The action spectrum $spec([\{f_t\}])$ is the set of all actions $A_F (y, D)$, where $F$ is a compactly supported Hamiltonian that generates $\{f_t\}$ and $y$ is a fixed point of $f_1$.

It is known that for every closed symplectic manifold there exists a subadditive spectral invariant.

For a subadditive spectral invariant $c$ we can define the asymptotic spectral invariant as
\[ \sigma(\phi) = \lim_{k \to \infty} \frac {c(\phi^k)}{k}.\]
Every asymptotic spectral invariant is homogeneous and the stability property holds.
For an open displaceable set $U \subset M$, and an element $\phi \in \tHam(M)$ supported in $U$ we have
\[ \sigma(\phi) = -V^{-1} \cdot \Cal(\phi) ,\]
where $V = \int_M \omega^n$ and $\Cal(\phi)$ is the Calabi invariant of $\phi$ if we regard it as an element of $\tHam(U)$.

Denote
\[I(\phi,\psi):= | \sigma(\phi \psi) - \sigma(\phi)-\sigma(\psi) |.\]
It is known that
\[I \leq \min(q(\phi),q(\psi)),\]
where $q(\phi) = c(\phi) + c(\phi^{-1})$.
It is also known that for a displaceable set $U$, one has
\[ \sup q(\phi) < \infty,\] where the supremum runs over all $\phi \in \tHam(M)$ supported in $U$.
We denote this value by
\[w(U) = \sup q(\phi).\]

For the proofs of these facts and for a further discussion on spectral invariants see \cite{3}.

\begin{proposition}
\label{p2}
Let $\phi, \psi \in \tHam(M)$. Assume that $\phi \psi = \psi \phi$ and that $\phi$ is supported in a displaceable set $U \subset M$.
Then $I(\phi,\psi) = 0$.

\end{proposition}
\begin{proof}
\[ \sigma(\phi \psi) = \frac {\sigma(\phi^k \psi^k) }{k} = \frac {k \sigma(\phi) + k \sigma(\psi) + C(k)} {k}, \]
where $C(k)$ is a constant depending on $k$ with $| C(k) | \leq w(U)$.
We get that
\[ \sigma(\phi \psi) = \sigma(\phi) + \sigma(\psi) + \frac{C(k)}{k} \xrightarrow{k \to \infty} \sigma(\phi) + \sigma(\psi). \]

\end{proof}

\section{Proof that $\chi$ is open}
Let $M$ be a closed symplectic manifold with $H^1(M) = 0$.

\begin{proposition}
\label{p1}
For $\phi \in \tHam(M)$, if $\sigma(\phi) \neq 0$ then $\{ \phi^n \}_{n \in \Z}$ is undistorted.
\end{proposition}
\begin{proof}
It is known that (see \cite{3})
\[ d(\Id, \phi) \geq | \sigma(\phi) | .\]
The claim thus follows in view of the homogeneity of $\sigma$.
\end{proof}

Put
\[\chi = \{ \phi \in \tHam(M) : \sigma(\phi) \neq 0 \}.\]

\begin{theorem}
\label{t4}
The set $\chi \subset \tHam(M)$ is $C^1$-open.

\end{theorem}

This is an easy consequence of the following.

\begin{theorem}
\label{t6}
The function $\sigma : \tHam(M) \to \R$ is $C^1$-continuous.
\end{theorem}
\begin{proof}
From the stability property of $\sigma$ together with the bi-invariance of the Hofer metric, we get that
it is enough to show that for every $\epsilon > 0$ if $\phi$ is $C^1$-close enough to the identity then there exists a Hamiltonian function $H$ that generates a representative such that
\[ \max(|H_t|) < \epsilon \]
for each $t$.

Let us recall some facts about symplectomorphisms which are $C^1$-close to the identity.
Let $\Delta \subset (M \times M, -\omega \oplus \omega)$ be the diagonal.
There is a symplectomorphism $\Psi$ from a neighborhood of the diagonal
\[ N(\Delta) \subset M \times M\]
to a neighborhood of the zero section
\[N(M_0) \subset T^*M\]
with the symplectic form $d\lambda_{can}$ defined on $T^*M$.
For a $C^1$-small Hamiltonian diffeomorphism $f$, the image $\Psi(graph(f))$ is the graph of an exact 1-form $dF$ (see \cite{1}).

For a smooth path of exact 1-forms $dG_t$ such that for each $t$, $dG_t$ is close enough to the zero section and $G_0 = 0$, there exist a Hamiltonian isotopy $\{g_t\}$ such that $graph(g_t) = \Psi^{-1}(graph(dG_t))$.
In addition, every loop of exact 1-forms $dG_t$ is homotopic to the zero section by the homotopy
\[ \{d(s \cdot G_t)\}_{s \in [0,1]}.\]
This proves the following
\begin{proposition}
\label{p3}
Let $f_t$ and $g_t$ be two paths of Hamiltonian diffeomorphisms with $f_0=g_0=\Id$ and $f_1=g_1$ that are $C^1$-close enough to the identity. Then they are homotopic with fixed end points.
\end{proposition}

Let $\{g_t\}$ be a representative of $\phi \in \tHam(M)$ which is $C^1$-close to $\Id$.
We get that $\Psi(graph(g_t)) = graph(dG_t)$ for some $G_t:M \to \R$.
Denote $F = G_1$. There is a Hamiltonian isotopy $\{f_t\}$, such that $\Psi(graph(f_t)) = graph(d(t \cdot F))$.
From Proposition \ref{p3}, we get that $\{f_t\}$ is a representative of $\phi$.

\textbf{Claim:} $\| \frac{\partial f_t}{\partial t} \| \to 0$ in $C^0$ as $\{ f_t \} \to \Id$ in $C^1$.

Before proving the claim, we use it to complete the proof of Theorem \ref{t6}:  By the claim, there exists a normalized Hamiltonian $H$ that generates $\{ f_t \}$  which is a representative of $\phi$, such that $\| sgrad H_t \|$ is arbitrarily small. Since $M$ is compact, there exists a constant $K$ such that
\[ | H_t(x) | < K \cdot \sup_{y \in M} \| sgrad H_t(y) \| \qquad \text{for all } x \in M \text{ and } t \in [0,1]. \]
This proves Theorem \ref{t6}.

\begin{it}Proof of the claim. \end{it}
Since $\Psi(graph(f_t)) = graph(d(t F))$, there exists a path of diffeomorphisms $h_t:M \to M$ such that
\[\Psi \circ gr_{f_t} = d(t \cdot F) \circ h_t ,\]
where $gr_{f_t}: M \to M \times M$ is defined by
\[gr_{f_t}(x) = (x, f_t(x)).\]
Denote by $\pi_1,\pi_2 : M \times M \to M$ the projections to the first and second factor respectively.
We get that
\[ \Id = \pi_1 \circ \Psi^{-1} \circ d(tF) \circ h_t, \]
\[ f_t = \pi_2 \circ \Psi^{-1} \circ d(tF) \circ h_t .\]
By differentiating both identities in $t$, one sees that $\| \frac{\partial f_t}{\partial t} \|$ is the sum of terms that tend to $0$ in $C^0$ as $\{f_t\} \to \Id$ in $C^1$ and $F \to 0$ in $C^0$. (Note that for $v \in T graph(f_t)$, $\frac{\|\pi_{1*}v\| - \|\pi_{2*}v\|}{\|v\|}$ is arbitrarily small.)

\end{proof}

\section{Obstruction to the existence of a partially hyperbolic symplectomorphism}

The next theorem provides an obstruction to the existence of a partially hyperbolic symplectomorphism. It will enable us to give examples of manifolds that do not admit partially hyperbolic symplectomorphisms.
From Theorem \ref{t8} we get that for these manifolds the set of elements whose time-1-map have elliptic periodic points is $C^1$-dense in $\tHam(M)$.

The idea of the obstruction has been announced by Bennequin (oral communication) and is presented in \cite{6}.

\begin{theorem}
\label{t1}
Let $M^{2n}$ be a closed symplectic manifold with non vanishing top Chern class, and $f \in Symp(M)$ a partially hyperbolic Hamiltonian diffeomorphism. Then there exists an even factorization $c(M) = \alpha \beta$ of the full Chern class of $M$.
\end{theorem}

\begin{theorem}
\label{t2}
Let $M^{2n}$ be a closed symplectic manifold and suppose that there is an isotropic subbundle L, i.e. $L \subset L^{\omega}$, and $rank(L) =i$.
Then there exists a factorization of the full Chern class $c(M) = \alpha \beta$ where $\alpha$ and $\beta$ are integral classes, $\alpha$ has only terms of even degree, $\deg(\alpha) \leq i$ and $\deg(\beta) \leq n-i$.
\end{theorem}

\begin{proof}[Proof that Theorem \ref{t2} implies Theorem \ref{t1}]
There exists a constant $Q > 0$ such that for all $v_1,v_2 \in TM$,
\[\omega(v_1,v_2) \leq Q \| v_1 \| \| v_2 \|.\]

For $x \in M$ and $u_1,u_2 \in E^s_x$,
\[ | \omega(u_1,u_2) | = | \omega(d_x f^k(u_1),d_x f^k(u_2)) | \leq \]
\[ \leq Q \| d_x f^k (u_1) \| \| d_x f^k (u_2) \| \leq b^2 \beta^{-2k} Q \|u_1\| \|u_2 \| \xrightarrow{k \to \infty} 0.\]
We get that $\omega(u_1,u_2) = 0$, so $E^s \subset (E^s)^ \omega$.
Hence $E^s$ is an isotropic subbundle, and so there exists a factorization $c(M) = \alpha \beta$, where $\alpha$ and $\beta$ are integral classes and $\alpha$ has only terms of even degree. Let us prove that $\deg(\alpha)$ cannot be zero or $n$. This will show that $c(M) = \alpha \beta$ is an even factoriztion.

Since $f$ is symplectic, we can assume that $rank(E^u) = rank(E^s) > 0$.
Denote $rank(E^s) = i$.
Note that on the one hand $\deg(c(M)) = n$ because $c_n(M) \neq 0$.
On the other hand,
\[ \deg(c(M)) \leq \deg(\alpha) + \deg(\beta) \leq i + (n-i) = n.\]
Hence all inequalities are actually equalities and $\deg(\alpha) = i > 0$.

If $\deg(\alpha) = n$, we get that $rank(E^s) = rank(E^u) = n$.
Hence $rank(E^c) = 0$.
On the other hand $f$ is isotopic to $\Id$, so $rank(E^c) > 0$ (see \cite{10}) and this is a contradiction.
We get that $0 < \deg(\alpha) < n$.
This completes the proof.

\end{proof}

\begin{proof}[Proof of Theorem \ref{t2}]

Let $J$ be a compatible almost complex structure.
The subbundle $L \oplus JL$ is symplectic and $L,JL$ are Lagrangian subbundles of $L \oplus JL$.
The subbundle $L \oplus JL$ is also isomorphic to the complexification of $L$.
From \cite[chapter 15]{7} we get that in $H^*(M;\Z)$ the odd Chern classes of $L \oplus JL$ are of order 2.
Hence when passing to rational coefficients the odd Chern classes vanish. We get that $c(L \oplus JL)$ has only terms of even degree.
We can write
\[TM = (L \oplus JL) \oplus (TM / (L \oplus JL)), \]
where the subbundle $(TM / (L \oplus JL))$ is also symplectic.
Put $\alpha = c(L \oplus JL)$ and $\beta = c(TM / (L \oplus JL))$.
Recall that Chern classes are integral classes.
This completes the proof.

\end{proof}

\section{$C^1$-Generic elements generate undistorted cyclic subgroups}

Let $M$ be a closed symplectic manifold with $H^1(M) = 0$, and let $\sigma$ be an asymptotic spectral invariant.

The following theorem shows that if the manifold has the property that a $C^1$-generic Hamiltonian diffeomorphism has an elliptic periodic point, then the set $\chi$ is $C^1$-dense in $\tHam(M)$.

\begin{theorem}
\label{t3}

Let $M$ be a closed symplectic manifold with $H^1(M) = 0$, and $\phi \in \tHam(M)$ such that its time-1-map has an elliptic periodic point.
Then for every $C^1$-open neighborhood $\mathcal{U} \subset \tHam(M)$ of $\phi$, there exists $\psi \in \mathcal{U}$ such that $\sigma(\psi) \neq 0$.

\end{theorem}

\begin{proof}[Proof of Theorem \ref{t5}]
Theorem \ref{t4} and Theorem \ref{t3} show that if a $C^1$-generic element has the property that its time-1-map has an elliptic periodic point then the set $\chi$ is $C^1$-open and dense in $\tHam(M)$. From Proposition \ref{p1} we get that
\[ \chi \subset  \{ \phi  \in \tHam(M) : \{\phi^n\}_{n \in \Z} \text{ is undistorted} \} .\]
Finally, from Theorem \ref{t1} we get that if $c_n(M) \neq 0$ and there is no even factorization then there are no partially hyperbolic symplectomorphisms and hence, by Theorem \ref{t8}, a $C^1$-generic Hamiltonian diffeomorphism has an elliptic periodic point.

\end{proof}

The rest of this section is dedicated to the proof of Theorem \ref{t3}, in which we follow the construction in \cite{5}.

The idea of the proof is to first construct an element $[\{ g_t\}] \in \tHam(M)$ which is $C^1$-close to $\phi$ and such that there exists a small open set $U \subset M$ and an integer $k \in \N$ such that $g^k | _U = \Id$ and $g^j(U) \cap U = \emptyset$ for all $j < k$.

The second step is to perturb $g$ inside $U$ in order to change the value of the asymptotic spectral invariant.

We start with the following lemma.

\begin{lemma}
\label{l1}

Let $[\{f_t\}] \in \tHam(M)$, and denote $f=f_1$. Let $p \in M$ be an elliptic $l$-periodic point of $f$. Then for any $C^1$-open neighborhood $\mathcal{U}$ of $[\{f_t\}]$ and any open neighborhood $V \subset M$ of $p$, there exists an element $[\{g_t\}] \in \tHam(M)$ and $\delta_1 > \delta_2 > 0$ such that $B_{\delta_1}(p)$ lies in a Darboux chart around $p$ and
\begin{enumerate}
\item $[\{g_t\}] \in \mathcal{U}$.
\item $B_{\delta_1}(p) \subset V$.
\item $g$ agrees with $f$ on the orbit of $p$, i.e. $g^i(p) = f^i(p), \forall i \in \{1,\ldots,l\}$.
\item $g$ agrees with $f$ outside the length $l$-orbit of $B_{\delta_1}(p)$, i.e.
 \[g | _{M \setminus \cup_{i=1}^l f^i(B_{\delta_1}(p))} = f | _{M \setminus  \cup_{i=1}^l f^i (B_{\delta_1}(p))}.\]
\item $g^l | _{B_{\delta_2}(p)} = T$, where $T$ is linear with simple, non real eigenvalues of the form $e^{\alpha_j 2 \pi \sqrt{-1}}$ with $\alpha_j$ rational. In this case we say that $T$ has eigenvalues with rational angles.

\end{enumerate}

\end{lemma}

\begin{proof}

The idea of the proof is to perturb the generating function of $f$. We divide the proof into three steps.
The first step is to construct a symplectomorphism that would be arbitrarily $C^1$-close to $f$ and such that all the properties asked for the time-1-map $g$ hold, except possibly that its eigenvalues are not rational angles.
The second step is to do another perturbation to get a symplectomorphism $g$ that satisfies all the conditions of the time-1-map $g$ in the lemma.
The last step will be to define a Hamiltonian isotopy from $f$ to $g$, and define $[\{g_t\}]$ to be the juxtaposition of $[\{f_t\}]$ and this Hamiltonian isotopy.

Let us begin with a simpler case.
Let $f:\R^{2n} \to \R^{2n}$ be a symplectomorphism with $f(0) = 0$.
Consider the symplectic matrix $df(0) : \R^{2n} \to \R^{2n}$.
We wish to construct a symplectomorphism $g$ such that for a small $\delta > 0$,
\[ g |_{B_\delta(0)} = df(0),\]
\[ g |_{\R^{2n} \backslash B_{2\delta}(0)} = f .\]

Recall that in a small neighborhood of $0$, there exists a generating function for $f$.
Denote $f(p_1,q_1) = (p_2,q_2)$.
Let $S:\R^{2n}(q,q') \to \R$ be the generating function of $f$, i.e.
\[ \frac{\partial S}{\partial q}(q_1,q_2) = -p_1 ,\]
\[ \frac{\partial S}{\partial q'}(q_1,q_2) = p_2 .\]
Since $f(0) = 0$ we can write
\[ S = \left< q, M_1 q \right> + \left< q, M_2 q' \right> + \left< q', M_3 q' \right> + k(q,q') ,\]
where $M_1,M_2,M_3$ are matrices and $k(q,q')$ has only terms of order greater than two.
Define a smoothened step function $a_\delta : \R \to \R$
\[ a_\delta(x) = \bigg\{ \begin{matrix} 0 & x<\delta \\ 1 & x>2\delta . \end{matrix} \]
Write
\[ \tilde{S} = \left< q, M_1 q \right> + \left< q, M_2 q' \right> + \left< q', M_3 q' \right> + a_{\delta}(\|(q,q') \|) k(q,q') .\]
We can choose $\delta$ so small that there exists a symplectomorphism $g$ such that $\tilde{S}$ is its generating function in a neighborhood that contains $B_{2\delta}(0)$.
Note that
\[g |_{B_\delta(0)} = df(0),\]
\[ g |_{\R^{2n} \backslash B_{2\delta}(0)} = f.\]
In order for $g$ to be $C^1$-close to $f$, we need $\tilde{S}$ to be $C^2$-close to $S$.
One can check that the norm of the difference between the second derivative of $\tilde{S}$ and the second derivative of $S$ is $O(\delta)$. So we can choose $\delta$ so small that $g$ is arbitrarily $C^1$-close to $f$. Note that this construction fails if we would try to make $g$ be $C^k$-close to $f$, for $k > 1$.

Let us return to the proof of the lemma.
For each $0 \leq i < l$ choose a Darboux chart $U_i$ around $f^i(p)$ such that $U_i \cap U_j = \emptyset$, and $f^i(p)$ is identified with $0 \in \R^{2n}$. Take a ball $B_0 \subset U_0$ such that $f^i(B_0) \subset U_i$, and $f^l(B_0) \subset U_0$. Since $f^i(B_0)$ and $f^{i+1}(B_0)$ are subsets of Darboux charts, we can treat $f |_{f^i(B_0)} : f^i(B_0) \to f^{i+1}(B_0)$ as a symplectomorphism between subsets of $\R^{2n}$. From the construction above, we get a symplectomorphism $\bg$ which is a linear map in a small ball inside $f^i(B_0)$ for each $i$, $\bg=f$ outside a larger ball inside $f^i(B_0)$ for each $i$, and $\bg(0) = 0$ for each $i$, that is $\bg^i(p) = f^i(p)$ for each $i$.
Note also that in a small ball $B$ inside $B_0$, $\bg^l: B \to \bg^l(B)$ is the product of all the matrices $d_{f^i(p)}(f |_{f^i(B_0)})$, so it is also linear. Denote this linear map by $\bar{T}$. We get that the symplectomorphism $\bg$ satisfies almost all the properties asked for the time-1-map in the lemma. The only property that possibly does not hold is that the eigenvalues of $\bar{T}$ are with rational angles.

Our next task is to find a $C^\infty$ perturbation $g$ such that $g^l$ restricted to a small enough ball inside $B_0$ is a matrix $T$ whose eigenvalues are with rational angles. Since $p$ is an elliptic point of $f^l$, we can choose such a symplectic matrix $T$ which is close to $\bar{T}$. Denote by $H_1$ the Hamiltonian function defined on $f^{l-1}(B_0)$ that generates $d_{f^{l-1}(p)}(f |_{f^{l-1}(B_0)})$ as its time-1-map. Let $Q$ be the symplectic matrix such that $\bar{T} Q = T$. Let $H_2$ be the Hamiltonian function such that $H_1 + H_2$ generates $d_{f^{l-1}(p)}(f |_{f^{l-1}(B_0)}) Q$ as its time-1-map. Choose a cutoff function $a$ supported in a small ball inside $f^{l-1}(B_0)$, and define the Hamiltonian function of the perturbed symplectomorphism in $f^{l-1}(B_0)$ to be $H_1 + a \cdot H_2$. Denote this new symplectomorphism by $g$. Note that outside a small ball inside $f^{l-1}(B_0)$, $g=\bg$, hence $g$ is a well defined symplectomorphism of $M$.  Since $H_2$ can be chosen arbitrarily small, we get that $g$ is arbitrarily $C^\infty$-close to $\bg$. Recall that
\[\bar{T} = \prod_{i=0}^{l-1} d_{f^i(p)}(f |_{f^i(B_0)}) ,\]
so we get that in a small ball inside $B_0$,
\[ g^l = \left( \prod_{i=0}^{l-2} d_{f^i(p)}(f |_{f^i(B_0)}) \right) d_{f^{l-1}(p)}(f |_{f^{l-1}(B_0)}) Q = \bar{T} Q = T .\]

Hence we can construct a symplectomorphism $g$ such that it is arbitrarily $C^1$-close to $f$, and it satisfies all the conditions in the lemma.

Since $g$ is $C^1$-close to $f$, we can construct a path of symplectomorphisms from $f$ to $g$ such that all symplectomorphisms in the path are $C^1$-close to $f$ (see \cite[Theorem 10.1 and its proof]{1}).
From the fact that $H^1(M) = 0$, we get that this path is a Hamiltonian isotopy.
Define $[\{g_t\}]$ to be the juxtaposition of $\{f_t\}$ and this Hamiltonian isotopy.

\end{proof}

\begin{proof}[Proof of Theorem \ref{t3}]
By Lemma \ref{l1} we find $[\{g_t\}] \in \tHam(M)$ that satisfies all the conditions in the lemma.
If $\sigma([\{g_t\}]) \neq 0$ then we are done, so suppose $\sigma([\{g_t\}]) = 0$.
Since all the eigenvalues of $T$ have rational angles, there is an integer $q$ (the smallest common multiple of the denominators) such that
\[g^{ql}|_{B_{\delta_2}(p)} = \Id .\]
Let $k \in \N$ be the smallest number such that $g^k|_{B_{\delta_2}(p)} = \Id$.
There exists $x \in B_{\delta_2}(p)$ such that $g^j(x) \neq x$ for all $0 < j < k$.
By continuity there is a ball $B \subset B_{\delta_2}(p)$ around $x$, such that
\[ g^j(B) \cap B = \emptyset ,\]
for all $0 < j < k$. We can choose $B$ such that the open set $\cup_{j=1}^k g^j(B)$ is displaceable.

Let $H : M \to [0,\infty)$ be a non-vanishing time independent Hamiltonian function supported in $B$. For $\epsilon > 0$, let $\{h'^\epsilon_t\}$ be the Hamiltonian isotopy generated by $\epsilon \cdot H$.
Put
\[ h^\epsilon_t(x) = \left\{ \begin{matrix} g^j \circ h'^\epsilon _t \circ g^{-j}(x) & x \in g^j(B), j=0,\ldots,k-1 \\
                                x & \text{otherwise} \end{matrix} \right. \]
Since $\cup_{j=1}^k g^j(B)$ is displaceable,
\[ \sigma([\{h^\epsilon_t\}]) = -\frac{1}{\int_M \omega^n} \Cal(h^\epsilon) = -\frac{k}{\int_M \omega^n} \Cal(h'^\epsilon) < 0.\]
The important part of this construction is that we get that the time-1-maps commute, i.e. $h^\epsilon _1 \circ g_1 = g_1 \circ h^\epsilon _1$.

\vspace{4 mm}

\textbf{Claim:} For a small enough $\epsilon$,
\[ [\{ g_t h^\epsilon_t\}] = [\{ h^\epsilon_t g_t\}].\]

\textbf{Proof:}  For small enough $\epsilon$, the path $\{ g_t h^\epsilon_t g_t^{-1} (h^\epsilon_t) ^{-1} \}$ is arbitrarily $C^1$-close to $\Id$. From this and from Proposition \ref{p3} we get that
\[ [  \{ g_t h^\epsilon_t g_t^{-1} (h^\epsilon_t) ^{-1} \} ] = \Id .\]
This completes the proof of the claim.

\vspace{4 mm}

From Proposition \ref{p2} we get that
\[ \sigma([ \{ g_t h^\epsilon_t \}]) = \sigma([\{g_t\}]) + \sigma([\{h^\epsilon_t\}]) = \sigma([\{h^\epsilon_t\}]) < 0 .\]
We can choose $\epsilon$ so small that $[\{g_t \circ h^\epsilon_t\}] \in \mathcal{U}$. This completes the proof.

\end{proof}

\section{$\tHam$ vs. $\Ham$}

Until now we discussed the notion of distortion of cyclic subgroups of $\tHam(M)$. One can ask if the same construction works if one considers undistorted cyclic subgroups of $\Ham(M)$ equipped with Hofer's metric (also denoted $d$).

Let $\pi : \tHam(M) \to \Ham(M)$ be the projection. Since $\pi$ is continuous and open, we get that if a set $S \subset \tHam(M)$ is open or dense in $\tHam(M)$, then $\pi(S) \subset \Ham(M)$ is open or dense respectively. In the case where $\sigma$ descends to $\Ham(M)$ we get that $\pi(\chi) \subset \Ham(M)$ is a $C^1$-open and dense subset of the set of Hamiltonian diffeomorphisms that generate undistorted cyclic subgroups. In the case where $\sigma$ descends our results thus extend to $\Ham(M)$.

 \begin{theorem}
 \label{t9}
 Let $M$ be a closed symplectic manifold such that
 \begin{enumerate}
 \item $H^1(M) = 0$.
 \item The top Chern class does not vanish, $c_n(M) \neq 0$.
 \item The full Chern class does not have an even factorization.
 \item There exists an asymptotic spectral invariant that descends to $\Ham(M)$.

 \end{enumerate}
 Then the set of elements in $\Ham(M)$ that generate undistorted cyclic subgroups has a $C^1$-open and dense subset.

 \end{theorem}

In \cite{8} McDuff gives conditions under which the asymptotic spectral invariants descend to $\Ham(M)$. In particular, we get that for $\CPn$ the asymptotic spectral invariants descend to $\Ham(\CPn)$.
In Example \ref{e1} we show that there is no even factorization of $c(T\CPn)$ and this proves Theorem \ref{t7}.

\section{Examples}
In this section we give examples of manifolds that satisfy the requirements of Theorem \ref{t5}. For these manifolds a $C^1$-generic element of $\tHam(M)$ generates an undistorted cyclic subgroup.

\begin{example}
Let $M = S^2$ be the 2-sphere. Note that $c_1(M) \neq 0$ and $H^1(M) = 0$, and obviously there is no even factorization of the full Chern class.
\end{example}

\begin{example}[Proof of Theorem \ref{t10}]
\label{e2}
Let $M^4$ be a 4-dimensional closed symplectic manifold such that $H^1(M) = 0$ and $M$ has a non-vanishing top Chern class, $c_2(M) \neq 0$.
Suppose that there is an even factorization $c(M) = \alpha \beta$.
This means that $0 < \deg(\alpha) < 2$ and $\alpha$ has only terms of even degree, which is impossible.
This proves Theorem \ref{t10}.
\end{example}

\begin{example}[Proof of Theorem \ref{t7}]
\label{e1}
Let $M = \CPn$. The full Chern class is
\[ c(M) = (1+a)^{n+1} - a^{n+1} , \]
where $a$ is a suitably chosen generator of $H^2(M)$ (see \cite{7}).
The top Chern class is
\[ c_n(M) = n+1 \neq 0.\]
Write
\[c(M) = C \prod_{i=1}^n (a - a_i)\]
where $C$ is a constant and
\[ a_i = \frac{1}{z_{n+1}^i - 1} ,\]
where $z_{n+1}$ is a primitive $n+1$-st root of unity.
Suppose that there is an even factorization $c(M) = \alpha \beta$. Note that we assume that $\deg(\alpha) + \deg(\beta) \leq n$ so when calculating the product $\alpha \beta$ the term $a^{n+1}$ will not appear, so in our calculation we can ignore the relation $a^{n+1} = 0$, and consider $c(M),\alpha,\beta$ as elements in the polynomials ring in the variable $a$.
Because we assume that $\deg(\alpha)>0$, there exists a root $x$ of the polynomial $c(M)$ that is a root of $\alpha$. Because $\alpha$ has only terms of even degree,  we get that $-x$ is also a root of $\alpha$ and hence a root of $c(M)$. Hence there are $1 \leq i_1,i_2 \leq n$ such that $a_{i_1} = - a_{i_2}$, i.e.
\[ \frac{1}{z_{n+1}^{i_1} - 1} = \frac{-1}{z_{n+1}^{i_2} - 1} \]
\[ z_{n+1}^{i_1} + z_{n+1}^{i_2} = 2. \]
Note that $|z_{n+1}^{i_1} | = 1$ and $| z_{n+1}^{i_2} | = 1$ but their sum is $2$ so we get that both are equal to $1$. This is a contradiction.
Hence $\CPn$ satisfies the requirements of Theorem \ref{t5}.
This together with Theorem \ref{t9} proves Theorem \ref{t7}.

\end{example}

\begin{example}
Let $M$ be the 1-point blow-up of $\mathbb{C}\mathbb{P}^3$.
We will show that $M$ satisfies the conditions of Theorem \ref{t5}.
The cohomology ring of $M$ is generated by 2 generators, the pull-back $a \in H^2(M)$ of a generator of $H^*(\mathbb{C}\mathbb{P}^3)$ and the \Poincare \; dual $b \in H^2(M)$ of the exceptional divisor, with the relations
\[ a b = 0, b^3 = a^3.\]

The full Chern class of $M$ is
\[ c(M) = 1 + 4a + 6a^2 + 6a^3 - 2b.\]
To see this let $\bar{a}$ be the corresponding generator of $H^*(\mathbb{C}\mathbb{P}^3)$, and write
\[ c(\mathbb{C}\mathbb{P}^3) = 1 + 4\bar{a} + 6\bar{a}^2 + 4\bar{a}^3 .\]
One can use this to compute the first and second Chern class of $M$ by a formula given in \cite[pp. 608--609]{11} and get that
\[ c_1(M) = 4a - 2b,\]
\[ c_2(M) = 6a^2.\]
We are only left with computing the top Chern class. Since the top Chern class is equal to the Euler class, in order to calculate it one needs to know the alternating sum of the Betti numbers. In our situation the odd cohomology groups vanish, so we only need to count the dimensions of the even cohomology groups. The pull back of $H^*(\mathbb{C}\mathbb{P}^3)$ contributes four even dimensions (generated by $1,a,a^2,a^3$), and by performing the blow-up we added an additional two even dimensions (generated by $b$ and $b^2$). Hence we get that the alternating sum of the Betti numbers is 6, and this gives us the final formula for the full Chern class.
See also \cite[Example 15.4.2(c)]{12}.

Suppose that $c(M) = \alpha \beta$ is an even factorization.
For a general $\alpha \in H^*(M)$ with even degrees and a general $\beta \in H^*(M)$, one can write
\[ \alpha = 1 + n_1 a^2 + n_2 b^2 .\]
\[ \beta = 1 + m_1 a + m_2 b.\]
Calculate
\[ \alpha \beta = (1+n_1 a^2) (1 + m_1 a) + n_2 b^2 + m_2 b + n_2 m_2 b^3 = c(M). \]
We get that $m_2 = -2, n_2 = 0$.
Since the coefficient of $b^3$ is zero,
\[ (1+n_1 a^2) (1 + m_1 a) = 1 + 4a + 6a^2 + 6a^3 := q(a).\]
We get from the factorization of $q$ above that there exist two roots $a_i,a_j$ of $q$ with $a_i = -a_j$.
The roots of the polynomial $q$ are
\[ a_1 \approx -0.38839, a_2 \approx -0.30581-0.57932 \sqrt{-1}, a_3 \approx -0.30581+0.57932 \sqrt{-1} ,\]
that is $q$ does not have roots such that $a_i = -a_j$, a contradiction.

Hence $M$ satisfies the conditions of Theorem \ref{t5}.

\end{example}

\begin{example}
Let $M = \mathbb{C}\mathbb{P}^2 \times \mathbb{C}\mathbb{P}^2$.
The cohomology of $M$ is generated by two generators $a,b$ with the relations $a^3 = b^3 = 0$.
The full Chern class is
\[ c(M) = (1 + 3a + 3a^2) (1 + 3b + 3b^2) .\]
Write a general even factorization
\[ c(M) = \alpha \beta.\]
Since $0 < \deg(\alpha) < 4$ and it is even, we get that $\deg(\alpha) = 2$.
From the equation $c(M) = \alpha \beta$ we can deduce that $\deg(\alpha)+\deg(\beta) \geq 4$. From the definition of an even factorization $\deg(\alpha) + \deg(\beta) \leq 4$ and hence $\deg(\beta) = 2$.
One can write
\[ \alpha = 1 + c_1 a^2 + c_2 b^2 + c_3 a b, \]
\[ \beta = 1 + d_1 a + d_2 a^2 + d_3 b + d_4 b^2 + d_5 a b .\]
Look at the equality $c(M) = \alpha \beta$. Each coefficient in $c(M)$ gives us an equation for the variables
\[ c_1,c_2,c_3,d_1,d_2,d_3,d_4,d_5.\]
Hence, we have 8 equations and 8 variables.
One can solve these equations and get two sets of solutions where each of them has non-integer values.
Since $\alpha$ and $\beta$ are integral classes, this is a contradiction.
Hence $c(M)$ has no even factorization, and $M$ satisfies the conditions of Theorem \ref{t5}.

\end{example}

\vspace{5 mm}

\textbf{Acknowledgement:}  This article was written under the guidance of Professor Leonid Polterovich. I also wish to thank Viktor Ginzburg and Amie Wilkinson for their helpful comments and suggestions. I wish to thank the anonymous referee for his/her thorough review and very useful comments.

School of Mathematical Sciences

Tel Aviv University

Tel Aviv 6997801, Israel

asafkisl@post.tau.ac.il


\begin{thebibliography}{5}
\bibitem{5}
  Bonnati Ch., Crovisier S., Vago G., Wilkinson A.,
  \emph{Local density of diffeomorphisms with large centralizers}.
  Annales Scientiﬁques de l`Ecole Normale Sup\'{e}rieure, t. 41($4^e$ s\'{e}rie), 925--954, 2009

\bibitem{6}
  Bonatti C., Díaz L., Viana M.,
  \emph{Dynamics beyond uniform hyperbolicity. A global geometric and probabilistic perspective}.
  Encyclopaedia of Mathematical Sciences, 102. Mathematical Physics, III. Springer-Verlag, Berlin, 2005

\bibitem{12}
 Fulton W.,
 \emph{Intersection theory}.
 Springer-Verlag, Berlin, Vol. 93, 1998

\bibitem{11}
  Griffiths P., Harris J.,
  \emph{Principles of algebraic geometry}.
  Wiley, New York, 1978

\bibitem{9}
  Hirsch M.,
  \emph{Differential topology}.
  Graduate Texts in Mathematics 33, Springer-Verlag, New York-Heidelberg, 1976

\bibitem{13}
  Horita V., Tahzibi A.,
  \emph{Partial hyperbolicity for symplectic diffeomorphisms}.
  Annales de l`institut Henri Poincar\'{e}, Analyse non-lin\'{e}aire, Volume 23, 641--661, 2006

\bibitem{8}
  McDuff D.,
  \emph{Monodromy in Hamiltonian Floer theory}.
  Comment. Math. Helv. 85.1 , 95--133, 2010

\bibitem{1}
  McDuff D., Salamon D.,
  \emph{Introduction to symplectic topology}.
  Oxford University Press, New York, 1998

\bibitem{7}
  Milnor J., Stasheff J.,
  \emph{Characteristic classes}.
  Annals of Mathematics Studies, No. 76. Princeton University Press, Princeton, N. J

\bibitem{2}
  Polterovich L.,
  \emph{The geometry of the group of symplectic diffeomorphisms}.
  ETH Lectures in Mathematics, Birkh\"{a}user, 2001

\bibitem{3}
  Polterovich L., Rosen D.,
  \emph{Function theory on symplectic manifolds}.
  Preprint, 2014

\bibitem{4}
  Saghin R., Xia Z.,
  \emph{Partial hyperbolicity or dense elliptic periodic points for $C^1$ generic symplectic diffeomorphisms}.
  Trans. Amer. Math. Soc., Volume 358, Number 11, 5119--5138, 2006

\bibitem{10}
   Shiraiwa K.,
   \emph{Manifolds that do not admit Anosov diffeomorphisms}.
   Nagoya Math. J. Volume 49, 111--115, 1973


\end{thebibliography}
\end{document}